\documentclass[amssymb, 11pt]{amsart}
\usepackage{latexsym}

\newlength{\standardunitlength}
\newtheorem{prop}{Proposition}[section]

\newtheorem{theorem}[prop]{Theorem}

\begin{document}

\title{Carries and a map on the space of rational functions}

\author{Jason Fulman}
\address{Department of Mathematics, University of Southern California, Los Angeles, CA 90089-2532, USA }
\email{fulman@usc.edu}

\date{June 8, 2023}

\thanks{The author was funded by Simons Foundation grants 400528 and 917224, and thanks Persi Diaconis
for infinitely many conversations about the mathematics of carries.}

\begin{abstract} A paper by Boros, Little, Moll, Mosteig, and Stanley relates properties of a map defined on the space of rational
functions to Eulerian polynomials. We link their work to the carries Markov chain, giving a new proof and slight generalization
of one of their results.
\end{abstract}

\maketitle

\section{Introduction}

Let $\Re$ denote the space of rational functions with complex
coefficients. The Taylor expansion at $x=0$ of $R \in \Re$ is written as
\[ R(x) = \sum_{n >> -\infty} a_n x^n,\] where $n >> -\infty$ denotes the
fact that the coefficients vanish for large negative $n$.

Boros, Little, Moll, Mosteig, and Stanley \cite{bor} study the map $\Phi_2: \Re \mapsto \Re$ defined by
\[ \Phi_2(R(x)) = \sum_{n
>> -\infty} a_{2n+1} x^n. \] This can be written explicitly as
\[ \Phi_2(R(x)) = \frac{R(\sqrt{x})-R(-\sqrt{x})}{2 \sqrt{x}} \] and arose
in a general procedure for the integration of rational functions
\cite{BMoll} \cite{vm}.

We work with the slightly more general map $\Phi_b: \Re \mapsto \Re$ defined
for positive integers $b$ by
\[ \Phi_b(R(x)) = \sum_{n
>> -\infty} a_{bn+(b-1)} x^n.\] Letting $\Phi_b^r$ denote $r$ iterations of $\Phi_b$, it is
clear that $\Phi_b^r=\Phi_{b^r}$,
which suggests possible connections with riffle shuffling or carries.

Let us restrict to rational functions in the class \[ \mathcal{A}:=
\left\{ \frac{P(x)}{(1-x)^n}: \text{P is polynomial of degree} \leq n-2
\right\}.\] The proof of Theorem \ref{mainlemma} of this paper shows that $\mathcal{A}$ is
stable under $\Phi_b$.

In what follows we use Eulerian polynomials $A_n(x)$. To define these, we say that
a permutation $\pi$ has a descent at position $i$ ($1 \leq i \leq n-1$) if $\pi(i)>\pi(i+1)$.
For example, $\underline{5} 1 3 \underline{6} 2 4$ has two descents. We let the Eulerian
number $A(n,k)$ be the number of permutations on $n$ symbols with $k$ descents, and define the Eulerian polynomial
 by \[ A_n(x)=\sum_{k=0}^{n-1} A(n,k) x^k.\] 

Our main result is Theorem \ref{asymp}. This calculates asymptotics of iterates of the
map $\Phi_b$, and was in \cite{bor} for the special case $b=2$. Our proof of Theorem
\ref{asymp} uses the carries Markov chain and is new even in that case. We mention
that Theorem \ref{asymp} resembles Theorem 5.1 of \cite{DF2}.

\begin{theorem} \label{asymp} Let $h(x)=h_{n-2}x^{n-2} + \cdots +h_1 x
+h_0$ be a polynomial of degree at most $n-2$. Then
\[ \lim_{r \rightarrow \infty} \frac{\Phi_b^r
\left( \frac{h(x)}{(1-x)^n} \right) }{b^{r(n-1)}} = \frac{A_{n-1}(x) \cdot
h(1)}{(1-x)^n (n-1)!}.\]
\end{theorem}

Let us next describe the carries Markov chain, which was introduced by Holte \cite{ho} and is
likely unfamiliar to many readers. Consider adding two 40-digit binary numbers (the top row,
in italics, comprises the carries). 

\[
\begin{array}{lllllllll}
{\it 1}& {\it 01110} & {\it 01000} & {\it 00001} & {\it 00111} & {\it
10111} & {\it 00000}& {\it 00111} & {\it 11100 }\\
 &10111&00110&00000&10011&11011&10001&00011&11010\\
 &10011&10101&11110&10001&01000&11010&11001&01111\\ \hline
1&01010&11011&11111&00101&00100&01011&11101&01001
\end{array}
\]
For this example, 19/40=47.5\% of the columns have a carry of 1. It is
natural to guess that when two large numbers are added, about half the time
there will be a carry. When three numbers are added, there can be a carry of
$0,1,2$. It is less obvious to guess the proportions. Is it $1/3,1/3,1/3$? 
It turns out that the limiting proportions are given in terms of Eulerian numbers:
\[ A(3,0)/3! = 1/6, A(3,1)/3! = 4/6, A(3,2)/3!=1/6.\]

More generally, if one adds $n$ integers in any base $b$, the carry must be in $\{0,1,\cdots,n-1\}$.
Further, if the
numbers are represented base $b$ and chosen ``at random'' with the
digits chosen independently and uniformly in $\{0,1,\cdots,b-1\}$,
the sequence of carries
$\kappa_0=0,\kappa_1,\kappa_2,\ldots$ working from the right
is a Markov chain taking values in
$\{0,1,2,\ldots,n-1\}$.

Let $K_b(i,j) = P(\kappa'=j|\kappa=i)$ denote an entry of the transition matrix
between successive carries $\kappa$ and $\kappa'$. Among other things,
Holte established the following properties:

\begin{description}

\item[(H1)]  For all $0 \leq i,j \leq n-1$, \[ K_b(i,j)=\frac{1}{b^n}
\sum_{r=0}^{j-\lfloor i/b\rfloor}
(-1)^r {n+1 \choose r} {n-1-i+(j+1-r)b \choose n}.\] 

\item[(H2)] The matrices $K_a$ and $K_b$ multiply according to the rule 
$K_a K_b=K_{ab}$ for all real $a$ and $b$.

\item[(H3)] The matrix $K_b$ has stationary vector $\pi$ (left eigenvector
with eigenvalue 1) independent of the base $b$:
\[ \pi(j)=\frac{A(n,j)}{n!},
\] where $A(n,j)$ is an Eulerian number. The $n!$ in the
denominator is to make the entries of the left eigenvector sum to 1.
For three other proofs of this property, see Section 6.2 of \cite{DFbook}.

\end{description}

The carries Markov chain is truly ubiquitous; see Chapter 6 of \cite{DFbook}
for a survey. Its many appearances in mathematics include: riffle shuffling
(\cite{DF1},\cite{DF2}), fractals \cite{holfrac}, additive combinatorics \cite{DSS},
and cohomology \cite{Isak}.

Finally, there is an appearance of the carries chain in a context very close to that
of the current paper: the Hilbert series of Veronese subrings. Let
\[ \sum_{k \geq 0} a_k x^k = \frac{h(x)}{(1-x)^{n+1}} \]
and suppose we are interested in every bth term $\{a_{bk}\}$. It is not hard to see that
\[  \sum_{k \geq 0} a_{bk} x^k = \frac{h^{\langle b \rangle} (x)}{(1-x)^{n+1}} \] 
for another polynomial $h^{\langle b \rangle} (x)$ of degree at most $n+1$. 
Brenti and Welker \cite{BW} show that the ith coefficient of $h^{\langle b \rangle} (x)$
satisfies  \[ h_i^{\langle b \rangle} = \sum_{j=0}^{n+1} M_b(i,j) h_j, \] where $M_b$
is a certain $(n+2) \times (n+2)$ matrix. From their formula it is not hard to see that
deleting the first and last rows and columns of $M_b$ and multiplying by $b^{-n}$ and taking the transpose
yields the carries matrix. Theorem \ref{mainlemma} of the current paper is similar
to this, but the connection with the carries matrix is even more striking, since one
doesn't need to delete rows and columns to obtain the carries matrix.

\section{Main results}

The first result relates the map $\Phi_b$ to the base $b$ carries chain. The proof
is in part inspired by the proof of Theorem 1.1 of \cite{BW}.

\begin{theorem} \label{mainlemma}
Let $h(x)=h_{n-2} x^{n-2} + \cdots + h_1 x + h_0$ be a polynomial of degree at
most $n-2$. Define $h_i'$ by
\[ \Phi_b \left(  \frac{h(x)}{(1-x)^n}  \right) = \frac{\sum_{i \geq 0} h_i' x^i}{(1-x)^n}.\]
Then \[ h_i' = b^{n-1} \sum_{j=0}^{n-2} K_b(j,i) h_j,\] where $K_b$ is the transition
matrix of the base $b$ carries chain for the addition of $n-1$ numbers.
\end{theorem}

\begin{proof} Define a sequence $a_n$ by \[ \sum_{n \geq 0} a_n x^n =
\frac{h_{n-2}x^{n-2} + \cdots +h_1 x +h_0}{(1-x)^n}.\] Since \[ \sum_{n
\geq -(b-1)} a_{n+(b-1)} x^n = \frac{h_{n-2}x^{n-2} + \cdots +h_1 x
+h_0}{x^{b-1}(1-x)^n },\] letting $\rho=e^{2 \pi i/b}$, it follows that
\begin{eqnarray*} \sum_{n \geq -(b-1)} a_{bn+(b-1)}x^{bn} & = &
\frac{1}{b}  \sum_{l=0}^{b-1} \frac{h(\rho^l x)}{(\rho^l
x)^{b-1}(1-\rho^lx)^n } \\
& = & \frac{1}{b}  \sum_{l=0}^{b-1} \frac{\rho^l \cdot h(\rho^l
x)}{x^{b-1}(1-\rho^lx)^n
} \\
& = & \frac{1}{b}  \sum_{l=0}^{b-1} \frac{\rho^l \left[ \frac{1-(\rho^l
x)^b}{1-\rho^lx} \right]^n h(\rho^l x)}{x^{b-1}(1-x^b)^n
} \\
& = & \frac{1}{b} \sum_{l=0}^{b-1} \frac{\rho^l}{x^{b-1}} \frac{\left[ 1+
\rho^lx+ \cdots+ (\rho^lx)^{b-1} \right]^n h(\rho^l x)}{(1-x^b)^n}.
\end{eqnarray*} Letting $C_b(i,n)$ denote the number of ways of writing
$i$ as a sum $y_1+\cdots+y_n$ with $y_i \in \{0,1,\cdots,b-1\}$, this
becomes
\begin{eqnarray*}
& & \frac{1}{b} \sum_{i \geq 0} \frac{C_b(i,n)
\sum_{l=0}^{b-1} \rho^l (\rho^l x)^i h(\rho^lx)}{x^{b-1}(1-x^b)^n}\\
& = & \frac{1}{b} \sum_{j =0}^{n-2} \sum_{i \geq 0}
\frac{C_b(i,n) h_j \sum_{l=0}^{b-1} (\rho^l x)^{i+j} \rho^l}{x^{b-1}(1-x^b)^n} \\
& = & \frac{1}{b} \sum_{j =0}^{n-2} \sum_{i \geq 0}
\frac{C_b(i-j,n) h_j x^i \sum_{l=0}^{b-1} \rho^{l(i+1)}}{x^{b-1}(1-x^b)^n} \\
& = & \sum_{j=0}^{n-2} \sum_{i \geq 0 \atop i=b-1 \ \text{mod \ } b}
\frac{C_b(i-j,n) h_j x^{i-(b-1)}}{(1-x^b)^n} \\
& = & \sum_{j=0}^{n-2} \sum_{i \geq 0 \atop b|i}
\frac{C_b(i-j+b-1,n) h_j x^i}{(1-x^b)^n} \\
& = & \sum_{j=0}^{n-2} \sum_{i \geq 0} \frac{C_b(ib-j+b-1,n) h_j
x^{ib}}{(1-x^b)^n}. \end{eqnarray*}

Summarizing, it has been shown that \begin{eqnarray*} \sum_{n \geq 0}
a_{bn+(b-1)} x^{bn} & = & \sum_{n \geq -(b-1)} a_{bn+(b-1)} x^{bn}\\
& = & \sum_{j=0}^{n-2} \sum_{i \geq 0} \frac{C_b(ib-j+b-1,n) h_j
x^{ib}}{(1-x^b)^n}.\end{eqnarray*} Replacing $x^b$ by $x$, it follows that
\[ \Phi_b \left( \frac{h(x)}{(1-x)^n} \right) = \frac{\sum_{i \geq 0}
h_i' x^i}{(1-x)^n},\] where  \[ h_i' = \sum_{j=0}^{n-2} C_b(ib-j+b-1,n)
h_j.\]

Note that $C_b(ib-j+b-1,n)=0$ if $i \geq n-1, j \leq n-2$. Thus $\sum_i
h_i' x^i$ is a polynomial of degree at most $n-2$. Note that $C_b(ib-j+b-1,n)$ is equal to the
coefficient of $x^{ib-j+b-1}$ in $(1+x+\cdots+x^{b-1})^n$. From page 140
of \cite{ho}, it follows that \[C_b(ib-j+b-1,n)= b^{n-1} K_b(j,i),\] where
$K_b$ is the transition matrix of the base $b$ carries chain for the
addition of $n-1$ numbers. This completes the proof.
\end{proof}

Given the connection between the carries Markov chain and the map $\Phi_b$,
it is now straightforward to prove Theorem \ref{asymp}.

\begin{proof} (Of Theorem \ref{asymp}) By the previous result, \[ \frac{\Phi_b^r
\left(\frac{h(x)}{(1-x)^n} \right)}{b^{r(n-1)}} = \frac{\Phi_{b^r}
\left(\frac{h(x)}{(1-x)^n} \right)}{b^{r(n-1)}} = \frac{\sum_{i=0}^{n-2}
\sum_{j=0}^{n-2} K_{b^r}(j,i) h_j x^i}{(1-x)^n}.\] 

From property (H2) in the introduction, $K_{b^r}(j,i)$ is the chance that the base $b$ carries
chain for the addition of $n-1$ numbers goes from $j$ to $i$ in $r$ steps. From property
(H3) in the introduction, the stationary distribution of the base $b$ carries chain is given by
$A(n-1,i)/(n-1)!$. Thus for all $j$, \[ \lim_{r \rightarrow \infty} K_{b^r} (j,i) = \frac{A(n-1,i)}{(n-1)!},\] which completes
the proof.
\end{proof}

\end{document}